\documentclass[12pt,reqno]{amsart}

\usepackage{amssymb,latexsym}

\usepackage{enumerate}

\usepackage[french,english]{babel}
\usepackage{amsmath}
\usepackage{graphicx}
\usepackage{amssymb}
\usepackage{bbm}
\usepackage{amsthm,mathtools}
\usepackage{ulem}
\usepackage{geometry}
\usepackage{tikz-cd}
\usepackage{mathrsfs}
\usepackage[colorinlistoftodos]{todonotes}
\usepackage{enumitem}
\usepackage{verbatim}
\usepackage[foot]{amsaddr}
\usepackage{dsfont}

\makeatletter

\@namedef{subjclassname@2010}{
	
	\textup{2020} Mathematics Subject Classification}

\makeatother
\newtheorem{thm}{Theorem}[section]
\newtheorem*{thm*}{Theorem}

\newtheorem{lem}[thm]{Lemma}

\theoremstyle{definition}

\numberwithin{equation}{section}

\newcommand{\F}{\mathcal F}

\usepackage{hyperref}
\hypersetup{hypertex=true,colorlinks=true,linkcolor=blue,anchorcolor=blue,citecolor=blue}
\newcommand{\newabstract}[1]{%
	\par\bigskip
	\csname otherlanguage*\endcsname{#1}%
	\csname captions#1\endcsname
	\item[\hskip\labelsep\scshape\abstractname.]
}

\begin{document}

	\baselineskip=17pt

	\title[Note on large  quadratic character sums]{Note on large quadratic character sums}

	\author{Zikang Dong\textsuperscript{1}}
\author{Yutong Song\textsuperscript{2,4}}
\author{Ruihua Wang\textsuperscript{3}}
    \author{Shengbo Zhao\textsuperscript{2}}
	\address{1.School of Mathematical Sciences, Soochow University, Suzhou 215006, P. R. China}
	\address{2.School of Mathematical Sciences, Key Laboratory of Intelligent Computing and Applications(Ministry of Education), Tongji University, Shanghai 200092, P. R. China}
   \address{3.School of Fundamental Sciences, Hainan Bielefeld University of Applied Sciences, Danzhou 578101, P. R. China}
    \address{4.	Graz University of Technology, Institute of Analysis and Number Theory, Steyrergasse 30/II, 8010 Graz, Austria}
	\email{zikangdong@gmail.com}

	\email{99yutongsong@gmail.com}
    \email{ruih.wan9@gmail.com}
	\email{shengbozhao@hotmail.com}

	\date{\today}
	
	\begin{abstract} 
		In this article, we investigate the conditional large values of quadratic Dirichlet character sums. We prove an Omega result for quadratic character sums under the assumption of the generalized Riemann hypothesis.
	\end{abstract}

	
	\maketitle
	
\section{Introduction}

Given a large prime number $q$ and a Dirichlet character $\chi\,(\textup{mod }q)$, we have the following P\'olya–Vinogradov inequality uniformly for all $x>0$:
\begin{equation*}
    \sum\limits_{n\leq x}\chi(n)\ll\sqrt{q}\log Q,
\end{equation*}
where $Q=q$ unconditionally, and $Q=\log q$ assuming the Generalized Riemann Hypothesis (GRH). According to a result of Paley \cite{Paley}, the conditional bound is optimal up to the choice of an implied constant.

We  define
\begin{equation*}
    \Delta_q(x):=\max_{\chi_0\neq\chi\,(\text{mod } q)} \left|\sum\limits_{n\leq x}\chi(n) \right|.
\end{equation*}
A thorough understanding of $\Delta_q(x)$ provides important information about the gap between the upper and lower bounds of character sums. Granville and Soundararajan studied this in detail in \cite{GS01}, distinguishing the results by comparing $x$ with $\exp(\sqrt{q})$ (Theorems 4 and 5 in \cite{GS01}). For simplicity, we call the character sum “short’’ when $x\leq\exp((\log q)^{\frac{1}{2}-\epsilon})$ and “long’’ when $x\geq\exp((\log q)^{\frac{1}{2}+\epsilon})$. We use $\epsilon$
to denote an arbitrarily small positive number, which may represent different values in different contexts.

The result of Granville and Soundararajan on 
\begin{equation*}
    \Delta_q(x):=\max_{\chi_0\neq\chi\,(\text{mod } q)} \left|\sum\limits_{n\leq x}\chi(n) \right|
\end{equation*}
has inspired many subsequent studies. In \cite{Mun}, Munsch showed that 
\begin{equation*}
    \Delta_q(x)\ge \Psi\bigg(x,\big(\tfrac14+o(1)\big)\frac{\log q\log_2q}{\max\{\log_2x-\log_3q,\log_3q\}}\bigg)
\end{equation*}
when $\log q\leq x\leq \exp(\sqrt{\log q})$. Here and throughout we use $\log_j(\cdot)$ as the $j$-th iteration of the logarithmic function. For $A\in \mathbb{R}$, Hough proved in \cite{Hou} that when $x=\exp(\tau\sqrt{\log q\log_2q})$, we have 
\[
\Delta_q(x)\ge \sqrt{x}\exp\bigg((1+o(1))A(\tau+\tau')\sqrt{\frac{\log X}{\log_2 X}}\bigg),
\]
where $\tau$ and $\tau'$ are numbers depending only on $A$. Tenenbaum and de la Bret\`eche showed in \cite{BT} that when $\exp((\log q)^{\frac12+\delta})\le x\le q$, we have 

\[\Delta_q(x)\ge \sqrt x\exp\bigg((\sqrt2+o(1))\sqrt{\frac{\log( q/x)\log_3(q/x)}{\log_2(q/x)}}\bigg).\]

In this paper, we focus on the asymptotic lower bound for real primitive character sums. Real characters often exhibit boundary behaviors and are hard to handle due to their irregular properties. For example, Paley \cite{Paley} showed that the P\'olya–Vinogradov inequality, when using real characters and assuming GRH, is optimal. Let $\mathcal{F}$ denote the set of all fundamental discriminants. In this paper, we are interested in 
$$\max_{X<|d|\le 2X\atop d\in\F}\sum_{n\le x}\chi_d(n).$$
Granville and Soundararajan studied this sum as an analogue of the general character sum in \cite{GS}. See Theorems 9--11 there. For the distribution and structure for larger values of this quantity, we refer to \cite{DWZ23,Lam24}.

Our main result is the following theorem. 
\begin{thm}\label{thm1.2}
Assume GRH. Let $\exp\left(4\sqrt{\log X\log_2X}\log_3X\right)\le x\le \exp((\log X)^{\frac12+\varepsilon})$.
Then we have
\[ \max_{X<|d|\le 2X\atop d\in\F}\sum_{n\le |d|/x}\chi_d(n)\ge \sqrt{\frac Xx}\exp\bigg((\tfrac{\sqrt2}{2}+o(1))\sqrt{\frac{\log X}{\log_2 X}}\bigg).\]
\end{thm}
This theorem extends Theorem 1.2 in \cite{lqcs}, and generalizes part of Theorem 3.1 in \cite{Hou}. We use the resonance
method developed by Hilberdink \cite{Hil} and Soundararajan \cite{Sound}. The technical tool Lemma \ref{lem2.2} is due to Darbar and Maiti \cite{DM}, which is much stronger than that (Lemma 4.1) of Granville and Soundararajan in \cite{GS}.

    \section{Preliminary Lemmas}\label{sec2}
    The following Fourier expansion for character sums was first showed by P\'olya.
 \begin{lem}\label{lem2.1}
 Let $\chi({\rm mod}\;q)$ be any primitive character and $0<\alpha<1$. Then we have
$$\sum_{n\le \alpha q}\chi(n)=\frac{\tau(\chi)}{2\pi i}\sum_{1\le|m|\le z}\frac{\overline\chi(m)}{m}(1-e(-\alpha m))+O(1+q\log q/z),$$
where $\tau(\chi):=\sum_{n\le q}\chi(n)e(n/q)$ is the Gauss sum and $e(a):=e^{2\pi ia}$.
\end{lem}
\begin{proof}
This is \cite[p.311, Eq. (9.19)]{MVbook}.
\end{proof}
The following conditional estimate for characters has a good error term in use.
    \begin{lem}\label{lem2.2}
	Assuming GRH. Let $n=n_0n_1^2$ be a positive integer with $n_0$ the square-free part of $n$.
	 Then for any $\varepsilon>0$, we obtain
	\begin{align*}
	\sum_{|d|\le X\atop d\in\F} \chi_{d}(n)=\frac{X}{\zeta(2)}\prod_{p|n}\frac{p}{p+1}\mathds{1}_{n=\square}+ O\left(X^{\frac12+\varepsilon}f(n_0)g(n_1)\right),
	\end{align*}
	where   ${\mathds{1}}_{n=\square}$ indicates the indicator function of the square numbers, and
    $$f(n_0)=\exp((\log n_0)^{1-\varepsilon}),\;\;\;\;g(n_1)=\sum_{d|n_1}\frac{\mu(d)^2}{d^{\frac12+\varepsilon}}.$$
\end{lem}
\begin{proof}
    This follows directly from Lemma 1 of \cite{DM}.
\end{proof}
On the one hand, it is clear that 
$$f(n_0)\le n_0^\varepsilon\le n^\varepsilon,\;\;\;\;g(n_1)\le n_1^\varepsilon\le n^\varepsilon.$$
On the other hand, if we denote the largest prime factor of $n$ by $P_+(n)$, then $n_0,n_1\le \prod_{p\le P_+(n)}p$.
So easily we have
$$f(n_0)\le\exp\big(P_+(n)^{1-\varepsilon}\big),\;\;\;\;\;g(n_1)\le\exp\big(P_+(n)^{\frac12-\varepsilon}\big).$$

The following lemma plays a key role in the proof of Theorem \ref{thm1.2}.
\begin{lem}\label{rmrn}    Let $Y$ be large and $\lambda=\sqrt{\log Y\log_2 Y}$. Define the multiplicative function $r$ supported on square-free integers and for any prime $p$:$$r(p)=\begin{cases}   \frac{\lambda}{\sqrt p \log p}, &  \lambda^2\le p\le \exp((\log\lambda)^2),\\   0, & {\rm otherwise.}\end{cases}$$
If $\log Y\ge \log W>3\lambda\log_2\lambda$, then we have
\begin{align*}   \sum_{m_1,n_1\le  W\atop (m_1,n_1)=1}\frac{m_1n_1r(m_1)(n_1)}{\max\{m_1,n_1\}^3}\sum_{d\le\frac{Y}{\max\{m_1,n_1\}}\atop(d,m_1n_1)=1}r(d)^2\Big/\prod_{p}&(1+r(p)^2)\\&\ge \exp{\bigg((2+o(1))\sqrt{\frac{\log Y}{\log_2 Y}}}\bigg).\end{align*}\end{lem}
\begin{proof}    This follows directly from Page 97 of \cite{Hou}.\end{proof}


\section{Proof of Theorem \ref{thm1.2}}\label{sec4}
Choose $z=\sqrt{|d|x}\log |d|$, by Lemma \ref{lem2.1} we have
\begin{align*}&\sum_{n\le |d|/x}\chi_d(n)\\&=\frac{\tau(\chi_d)}{2\pi i}\sum_{1\le|m|\le z}\frac{\chi_d(m)}{m}\big(1-e(-m/x)\big)+O(\sqrt{|d|/x})\\
&=\frac{\tau(\chi_d)}{2\pi i}\sum_{1\le|m|\le z}\frac{\chi_d(m)}{m}\big(1-c(m/x)\big)+\frac{\tau(\chi)}{2\pi }\sum_{1\le|m|\le z}\frac{\chi_d(m)}{m}s(m/x)+O(\sqrt{|d|/x}),\end{align*}
where
$$e(a):=e^{2\pi ia},\;\;c(a):=\cos2\pi a,\;\;s(a):=\sin 2\pi a.$$
Let $$C_d(z):=\sum_{|m|\le z}\frac{\chi(m)}{m}\big(1-c(m/x)\big),$$
and
$$S_d(z):=\sum_{|m|\le z}\frac{\chi_d(m)}{m}\big(1-c(m/x)\big),$$
If $\chi_d(-1)=1$, then
$C_\chi=0$.
If $\chi_d(-1)=-1$,  then $S_\chi(z)=0$, and
$$\Big|\sum_{n\le |d|/x}\chi_d(n)\Big|=\frac{\sqrt {|d|}}{2\pi}|C_d(z)|+O(\sqrt{|d|/x}).$$
Thus we have 
\begin{align}\label{polya}
\max_{X<|d|\le 2X\atop d\in\F}\Big|\sum_{n\le |d|/x}\chi_d(n)\Big|\ge\max_{X<|d|\le 2X\atop d\in\F}\frac{\sqrt {|d|}}{2\pi}|C_d(z)|+O(\sqrt{|d|/x}).\end{align}

Let $y=X^{\frac12-\delta}/(2\log z)^2$ and  $\lambda=\sqrt{\log y\log_2y}$, where $0<\delta<\frac{1}{4}$ is any fixed small number. We define the   multiplicative function  $r(\cdot)$ supported on square-free integers as in Lemma \ref{rmrn} by  $$r(p)=\begin{cases}
   \frac{\lambda}{\sqrt p \log p}, &  \lambda^2\le p\le \exp((\log\lambda)^2),\\
   0, & {\rm otherwise.}
   \end{cases}$$  We define the resonator $$R(d):=\sum_{n\leq y}r(n)\chi_d(n),$$and 
$$M_1(R,X):=\sum_{X<|d|\le 2X\atop d\in\F}R(d)^2,$$
$$M_2(R,X):=\sum_{X<|d|\le 2X\atop d\in\F}R(d)^2C_d(z)^2.$$
Then we have 
\[\max_{X<|d|\le 2X\atop d\in\F}C_d(z)^2\geq {\frac{M_2(R,X)}{M_1(R,X)}}.\]
For $M_1$, by Lemma \ref{lem2.2}, we have 
\begin{align*}
M_1(R,X)&=\frac{X}{\zeta(2)}\sum_{m,n\le y\atop mn=\square}r(m)r(n)\prod_{p|mn}\frac{p}{p+1}+O\Big(X^{\frac12+\varepsilon}\sum_{m,n\le y}r(m)r(n)\Big)\\\nonumber
&\le\frac{X}{\zeta(2)}\sum_{m,n\le y\atop mn=\square}r(m)r(n)+O\Big(X^{\frac12+\varepsilon}y\sum_{m\le y}r(m)^2\Big)\\
&=\frac{X}{\zeta(2)}\sum_{m\le y}r(m)^2+O\Big(X^{1-\delta+\varepsilon}\sum_{m\le y}r(m)^2\Big),
\end{align*}
For $M_2$, write $a_k:=(1-c(k/x))/m$, we have
\begin{align*}\label{eq:strtdt}
  &M_2(R,X)\\&=\frac{X}{\zeta(2)}\sum_{1\le|k|,|\ell|\le z}a_ka_\ell\sum_{ m,n\leq y\atop k\ell mn=\square} r(m)r(n)\prod_{p|mnk\ell}\frac{p}{p+1}+O\Big(X^{\frac12+\varepsilon}\sum_{ m,n\leq y\atop k,\ell\leq x}a_ka_\ell r(m)r(n)\Big)
  \\&=\frac{X}{\zeta(2)}\sum_{1\le|k|,|\ell|\le z}a_ka_\ell\sum_{ m,n\leq y\atop k\ell mn=\square}r(m)r(n)\prod_{p|mnk\ell}\frac{p}{p+1}+O\Big(X^{\frac12+\varepsilon}(\log z)^2y\sum_{ m\leq y}r(m)^2\Big)\\
  &=\frac{X}{\zeta(2)}\sum_{1\le|k|,|\ell|\le z}a_ka_\ell\sum_{ m,n\leq y\atop k\ell mn=\square}r(m)r(n)\prod_{p|mnk\ell}\frac{p}{p+1}+O\Big(X^{1-\delta+\varepsilon}\sum_{ m\leq y}r(m)^2\Big)\\
  &\ge\frac{2X}{\zeta(2)}\sum_{k,\ell\le z}a_ka_\ell\sum_{ m,n\leq y\atop mk=n\ell}r(m)r(n)\prod_{p\le X}\frac{p}{p+1}+O\Big(X^{1-\delta+\varepsilon}\sum_{ m\leq y}r(m)^2\Big)\\
  &\ge\frac{2X}{\zeta(2)}(\log X)^{-c}\sum_{k,\ell\le z}a_ka_\ell\sum_{ m,n\leq y\atop mk=n\ell}r(m)r(n)+O\Big(X^{1-\delta+\varepsilon}\sum_{ m\leq y}r(m)^2\Big),
\end{align*}
where we used $y=X^{\frac12-\delta}/(2\log z)^2$, $k\ell mn=\square$ implies $k\ell\ge0$, and
$$\prod_{p|k\ell mn}\frac{p}{p+1}\ge\prod_{p\le X}\frac{p}{p+1}\ge(\log X)^{-c}$$
for some absolute positive $c$. 
Now  we get 
\begin{align*}
   \max_{X<|d|\le 2X\atop d\in\F}C_d(x)^2
   &\geq {\frac{M_2(R,X)}{M_1(R,X)}}\\ 
   &\geq(\log X)^{-c}\sum_{k,\ell\le x}\sum_{ m,n\leq y\atop mk=n\ell}a_ka_\ell r(m)r(n)\Big/\sum_{m\le y}r(m)^2+O(X^{-\delta+\varepsilon})\\
   &\gg x^{-4}(\log X)^{-c}\sum_{k,\ell\le x/2}\sum_{ m,n\leq y\atop mk=n\ell}k\ell r(m)r(n)\Big/\sum_{m\le y}r(m)^2+O(X^{-\delta+\varepsilon}).
\end{align*}
Note that, $mk=n\ell$ implies $k=n_1g$ and $\ell=m_1g$ for some $g$, where $m_1=m/(m,n)$ and $n_1=n/(m,n)$. Thus we have
\begin{align*}&\sum_{k,\ell\le x/2}\sum_{ m,n\leq y\atop mk=n\ell}k\ell r(m)r(n)\\&=\sum_{m,n\le y}r(m)r(n)\sum_{g\le x/\max\{m_1,n_1\}}m_1n_1g^2\\&\gg x^3\sum_{m_1,n_1\le \min\{y,x/2\}\atop(m_1,n_1)=1}\frac{m_1n_1r(m_1)r(n_1)}{\max\{m_1,n_1\}}^3\sum_{d\le\frac{y}{\max\{m_1,n_1\}}\atop(d,m_1n_1)=1}r(d)^2.\end{align*}

On the other hand, trivially we have
$$\sum_{m\le y}r(m)^2\le\prod_{p}(1+r(p)^2).$$
By Lemma \ref{rmrn}, we have
\begin{align*}&\max_{X<|d|\le 2X\atop d\in\F}C_d(x)^2\\&\ge x^{-1}\exp{\bigg((2+o(1))\sqrt{\frac{\log y}{\log_2 y}}}\bigg)\\&=x^{-1}\exp{\bigg((2\sqrt{\tfrac12-\delta}+o(1))\sqrt{\frac{\log X}{\log_2 X}}}\bigg).\end{align*}
At last, by \eqref{polya} we get
$$\max_{X<|d|\le 2X\atop d\in\F}\Big|\sum_{n\le |d|/x}\chi_d(n)\Big|\ge\sqrt{\frac Xx}\exp{\bigg(({\tfrac{\sqrt2}{2}}+o(1))\sqrt{\frac{\log X}{\log_2 X}}}\bigg),$$
since $\delta>0$ is arbitrarily small.

	\section*{Acknowledgements}
	Z. Dong is supported by the Shanghai Magnolia Talent Plan Pujiang Project (Grant No. 24PJD140) and the National
	Natural Science Foundation of China (Grant No. 	1240011770). Y. Song is supported by the China Scholarship Council (CSC) for her study in Austria.

	\normalem


\begin{thebibliography}{99}
		
			




\bibitem{BT} de la Bret\`{e}che, R.; Tenenbaum, G. {\emph Sommes de G\'{a}l et applications}, {\it Proc. Lond. Math. Soc.}, {\bf 119} (2019), 104--134.
	
	
	

\bibitem{DM} Darbar, P.; Maiti, G. {\emph Large values of quadratic Dirichlet $L$-functions}, {\it Math. Ann.}, (2025), pp. 1--33.

\bibitem{DWZ23} Dong, Z.; Wang, W.;  Zhang, H. {\emph Structure of large quadratic character sums}, preprint, arXiv:2306.06355
\bibitem{lqcs} Dong, Z.;  Zhang, Y. {\emph Large quadratic character sums}, preprint, arXiv:2509.07651
\bibitem{GS} Granville, A.; Soundararajan, K. {\emph The Distribution of values of $L(1,\chi_d)$}, {\it Geom. Funct. Anal.}, {\bf 13} (2003), 992--1028.
 


\bibitem{GS01} Granville, A.; Soundararajan K. {\emph Large character sums}, {\it J. Amer. Math. Soc.}, {\bf 14} (2001), 365--397.

\bibitem{Hil} Hilberdink, T. {\emph An arithmetical mapping and applications to results for the Riemann zeta
function}, {\it Acta Arith.}, {\bf 139} (2009), 341--367.
 
\bibitem{Hou} Hough, B. {\emph The resonance method for large character sums}, {\it  Mathematika},  {\bf 59},(2013), 87-118



\bibitem{Lam24} Lamzouri, Y. {\emph The distribution of large quadratic character sums and applications}, {\it Algebra $\&$ Number Theory}, {\bf 18} (2024), 2091--2131.

\bibitem{MVbook} Montgomery H. L.;  Vaughan  R.C.
\emph{Multiplicative Number Theory I: Classical Theory},
Cambridge Studies in Advanced Mathematics, Vol. 97. Cambridge University Press, 2006.
\bibitem{Mun} Munsch, M.  {\emph The maximum size of short character sums}, {\it Ramanujan J.} {\bf 53} (2020), 27–-38.

\bibitem{Paley}
 Paley, R. E. A. C. {\emph A theorem on characters}, {\it J. Lond. Math. Soc.}, {\bf 1} (1932),
28--32. 

\bibitem{Sound}
 Soundararajan, K. {\emph Extreme values of zeta and $L$-functions}, {\it Math. Ann.}, {\bf 342} (2008),
67--86. 





		
	\end{thebibliography}
\end{document}